\documentclass[12pt, 14paper,reqno]{amsart}
\usepackage{amsmath,amsfonts,amssymb}
\usepackage{csquotes}
\theoremstyle{plain}
\newtheorem{theorem}{Theorem}
\newtheorem{lemma}{Lemma}
\newtheorem{corollary}{Corollary}
\newtheorem{proposition}{Proposition}
\theoremstyle{proof}
\theoremstyle{definition}

\newtheorem{remark}{Remark}
\numberwithin{lemma}{section}
\numberwithin{equation}{section}
\numberwithin{theorem}{section}
\numberwithin{remark}{section}
\numberwithin{proposition}{section}

\begin{document}
\title{Quadratic reciprocity and Some ``non-differentiable" functions}
\author{Kalyan Chakraborty}
\address{Harish-Chandra Research Institute, HBNI,
Chhatnag Road, Jhunsi, Allahabad 211 019, India}
\email{kalyan@hri.res.in}
\author{Azizul Hoque}
\address{Harish-Chandra Research Institute, HBNI, Chhatnag Road, Jhunsi, Allahabad 211 019, India}
\email{azizulhoque@hri.res.in}
\date{}
\maketitle
\begin{abstract}
Riemann's non-differentiable function and Gauss's quadratic reciprocity law have attracted the attention of many researchers. 
In \cite{RM} Murty and Pacelli gave an  instructive proof of the quadratic reciprocity via the theta-transformation formula and Gerver \cite{G1} was the first to give a proof of differentiability/non-differentiability of Riemnan's function. The aim here is to survey some of the work done in these two questions and concentrates more onto a recent work of the first  author along with Kanemitsu and Li \cite{K1}.
In that work \cite{K1} an integrated form of the theta function was utilised and the advantage of that is that while the theta-function $\Theta(\tau)$ is a dweller in the upper-half plane, its integrated form $F(z)$ is a dweller in the extended upper half-plane including the real line, thus making it possible to consider the behaviour under the increment of the real variable, where the integration is along the horizontal line.
\end{abstract}

\section{Introduction}
In the early part of the 19th century  many mathematicians believed that a continuous function has derivative in a reasonably large set.  A. M. Amp\'{e}re in his paper in 1806  tried to give a theoretical justification for this based of course on the knowledge at that time. In a presentation before the Berlin Academy on July 18 on 1872, K. Weierstrass kind of shocked the mathematical community by proving this assertion to be false! He presented a function which was everywhere continuous but differentiable nowhere. We will talk about  this function of Weierstrass in the later sections in some details. This example was first published by du Bois-Remond in 1875. Weierstrass also mentioned Riemann, who apparently had used a similar construction (without proof though!) in his own 1861 lectures. However it seems like neither Weierstrass nor Riemann was first to get such examples. The earliest known example is due to B. Bolzano, who in the year 1830 exhibited (published in the year 1922 after being discovered a few years earlier) a continuous nowhere differentiable function. Around 1860 the Swiss mathematician, C. Cell\'{e}rier discovered such a function but unfortunately it was not published then and could be published only in 1890 after his death. To know more about the interesting history and details about such functions, the reader is referred to the excellent Master's thesis of J. Thim \cite{JT}.

Riemann, as mentioned in the earlier paragraph, opined that the function,
$$
f(x) = \sum_{n=1}^{\infty} \frac{\sin n^2x}{n^2}
$$
is nowhere differentiable. K. Weierstrass (in 1872) tried to prove this assertion, but couldn't  resolve it. He could construct another example of a continuous nowhere differentiable function 
$$
\sum_{n=0}^{\infty} \cos (b^n\pi x)
$$
where $0<a<1$ and $b$ is a positive integer such that
$$
ab > 1 + 3/2 \pi.
$$
G.H. Hardy \cite{H} showed that Weierstrass function has no derivative at points of the form $\xi \pi$ with $\xi$ 
is either irrational or rational of the form $2A/(4B + 1)$ or $(2A +1)/(2B +2)$. Much later in 1970,  J. Gerver  \cite{G1} disproved Riemann's assertion by proving that his function is differentiable at points of the form $\xi \pi$ where $\xi$ is of the form $(2A +1)/(2B +1)$, with derivative $-1/2$.  Arthur, a  few years later in 1972,used Poisson's summation formula and properties of Gauss sums to deduce Gerver's result and thus established a link between Riemann's function and quadratic reciprocity (via Gauss sums). Interested reader can also look into two excellent expositions of Riemann's function by E. Neuenschwander  \cite{Er} and that of S. L. Segal  \cite{Sl} for further enhancement in knowledge regarding this problem.
This problem was explored by many other authors and among them a few references could be \cite{H, Ka, G2, G3} and \cite{It}.

In an interesting work in \cite{K1}, the authors observed that Riemann's function $f(x)$ is really an integrated form of the classical $\theta$- function. Then they make the link to quadratic reciprocity from an exposition of M. R. Murty and A. Pacelli \cite{RM},  who (following Hecke)  showed that the transformation law for the theta function can be used to derive the law of quadratic reciprocity. The goal of \cite{K1} was to combine these two ideas and derive both, the differentiability of $f(x)$ at certain points and the law of quadratic reciprocity at one go.

An identity of Davenport and Chowla arose our interest in Riemann's function. The identity is:
\begin{equation}\label{eqdavenport:2.14}
\sum_{n=1}^{\infty}\frac{\lambda(n)}{n}\psi(nx)=-\frac{1}{\pi}\sum_{n=1}^{\infty}\frac{\sin
2\pi n^2x}{n^2}.
\end{equation}
The notations are standard, i.e.
\begin{equation*}
\lambda(n)=(-1)^{\Omega(n)}
\end{equation*}
with $\Omega(n)$ denotes the total number of distinct prime factors of $n$. Also,
$$
\psi(x)=-\frac{1}{\pi}\sum_{n=1}^{\infty}\frac{\sin 2\pi nx}{n}
$$
 is the saw-tooth Fourier series, i.e. it is the Fourier series expansion of the `saw-tooth' function:
 $$
 f(x) = 
 \begin{cases} 1/2 (\pi - x), ~\mbox{ if }~~0< x \le 2\pi ; \\
                          f(x + 2\pi), ~~ \mbox{otherwise}. 
\end{cases}                          
$$
We would like to spare some discussion for this identity. On one hand in \eqref{eqdavenport:2.14} there is the Liouville function, a prime number-theoretic entity. On the other hand one has Riemann's example of an interesting function. The integrated identity can be derived from the functional equation only, but to differentiate it, one needs the estimate for the error term for the Liouville function. This is as deep as the prime number theorem and is known to be very difficult.
 
The situation is similar to Ingham's handling \cite{Ing} of the prime number theorem. First one applies the Abelian process (integration)
and then Tauberian process (differencing) which needs more information. A huge advantage of this process in \cite{K1} is that while the theta-function $\Theta(\tau)$ dwells  in the upper-half plane, its integrated form $F(z)$ is a dweller in the extended upper half-plane which includes the real line. This  makes it possible to consider the behaviour under the increment of the real variable, where the integration is along the horizontal line. The elliptic theta-function $\theta(s)=\Theta(-i\tau)$ is a dweller in the right half-plane $\{\sigma>0 \}$, where the integration is along the vertical line. In terms of Lambert series, an idea of Wintner to deals with limiting the behaviour of the Lambert series on the circle of convergence, i.e. radial integration. Here it corresponds to integration along an arc.
 
One can think of it as  two apparently disjoint aspects merging on the real line as limiting behaviours of zeta and that of theta-functions.
In \cite{K1}, the main observation was that the right - hand side may be viewed as the imaginary part of the integrated theta-series. It seems that the uniform convergence of the left-hand side and the differentiability of the right-hand side merge as the limiting behaviour of a sort of modular function and that of the Riemann zeta-function.

\section{Weierstrass's non-dirrefrentiable function}
We begin with the following function which is due to Weierstrass:
\begin{equation*}\label{wf1}
f(x)=\sum a^n\cos b^n\pi x.
\end{equation*}
In 1875, Weierstrass proved that $f(x)$ has no differential coefficient for any value of $x$ with restrictions  that $b$ is an odd integer,  
\begin{equation}\label{wf2}
0<a<1
\end{equation}
and 
\begin{equation}\label{wf3}
ab>1+\frac{3}{2}\pi.
\end{equation}
This result has been generalized by many mathematicians (for details see \cite{Darboux, Faber, Landberg, Lerch, Wiener}) by considering functions of more general forms 
\begin{equation}\label{wf4}
C(x)=\sum a_n\cos b_n x
\end{equation}
and 
\begin{equation}\label{wf5}
S(x)=\sum a_n\sin b_n x
\end{equation}
where $a_n$ and $b_n$ are positive, the series $\sum a_n$ is convergent, and the sequence $\{b_n\}$ increase steadily with more than a certain rapidity. 
In 1916, G. H. Hardy with a new idea, developed a powerful method to discuss the differentiability Weierstrass's function. This method is easy to apply to find very general conditions for the non-differntiability of the type of series (\ref{wf4}) and (\ref{wf5}). The known results concerning the series (\ref{wf4}) are, so far we are aware, are as follows:  K. Weierstrass gave the condition (\ref{wf3}) and only improvement to this is 
\begin{equation}\label{wf6}
ab>1+\frac{3}{2}\pi(1-a).
\end{equation}
This was due to T. J. Bromwich \cite{Bh}.
The conditions (\ref{wf3}) and (\ref{wf6}) debar the existence of a finite (or infinite) differential coefficient. For the non-existence of a finite differential coefficient, there are alternative conditions which were independently given by U. Dini, M. Lerch and T. J. Bromwich.  
The conditions given by U. Dini are 
\begin{equation}\label{dini1}
ab\geq 1, \hspace{1cm} ab^2>1+3\pi^2
\end{equation} 
and that are given by M. Lerch:
\begin{equation}\label{lerch}
ab\geq 1, \hspace{1cm} ab^2>1+\pi^2.
\end{equation}
Finally, T. J. Bromwich provided the following conditions for the same
\begin{equation}\label{lerch}
ab\geq 1, \hspace{1cm} ab^2>1+\frac{3}{4}\pi^2(1-a).
\end{equation}
All these conditions though supposed that $b$ is an odd integer. However, U. Dini \cite{Dini} showed that if the condition (\ref{wf3}) is replaced by 
\begin{equation}\label{dini2}
ab>1+\frac{3}{2}\pi\frac{1-a}{1-3a}
\end{equation}
or the condition (\ref{dini1}) by 
\begin{equation}\label{dini3}
ab>1, \hspace{1cm} ab^2>1+15\pi^2\frac{1-a}{5-21a}
\end{equation}
then the restriction ``odd" on $b$ may be removed. It is naturally in built  in the condition (\ref{dini2}) that $a<\frac{1}{3}$ and in the condition (\ref{dini3}) that $a<\frac{5}{21}$. 

The conditions (\ref{wf6})--(\ref{dini3}) look superficial though. It is hard to find any corroboration between these conditions as to why they really correspond to any essential feature of the problems arise in discussion of Weierstrass function. They appear merely as a consequence of the limitations of the methods  that were employed. There is in fact only one condition which suggests itself naturally and seems truly relevant to the situation at hand, namely:
 $$
 ab\geq  1.
 $$
The main results that were proved by G. H. Hardy  \cite{H} concerning Weierstrass function and the corresponding function defined by a series of sines and cosines, are summerized below. It is interesting to note that $b$ has no more restriction to be an integer in the next two  results. 
\begin{theorem}(Hardy)\label{HT1}
The functions
 $$
 C(x) = \sum a^n \cos b^n \pi x,\qquad S(x) = \sum a^n \sin b^n \pi x,
 $$
 (with $0 < a < 1, \  b > 1$) have no finite differential coefficient at any point whenever $ ab\geq 1$.
\end{theorem} 
 \begin{remark}\label{HR1}
The above Theorem \ref{HT1} is {\bf{not true}} if the word ``finite" is omitted. 
\end{remark} 
 \begin{theorem}(Hardy)\label{HT2}
Let $ab>1$ and so $\xi=\frac{\log (1/a)}{\log b}<1$. Then each of the functions in the previous theorem satisfy
$$
f(x+h)-f(x)=O\big( |h|^\xi\big),
$$ 
for each value of $x$. Neither of them satisfy 
$$
f(x+h)-f(x)=o\big( |h|^\xi\big),
$$
for any $x$. 
\end{theorem} 
Hardy proved these theorems in two steps. In the first step, he considered $b$ an integer and then in the second step, he extended his proof for general case. In the next two subsections, we give the outline of the proof of these theorems. 
\subsection{$b$ is an integer}
Let us substitute $\theta = \pi x$ and then the function of Weierstrass becomes  a Fourier series in $\theta$. Following which he 
defines a harmonic function $G(r, \theta)$ by the real part of the power series:
$$
\sum a_nz^n=\sum a_n r^n e^{ni\theta}.
$$
This series is convergent when $r < 1$. One further supposes that $G (r, \theta)$ is continuous for $r\leq 1$, and that 
$$
G(1, \theta) = g(\theta).
$$
Let us first recall some results concerning the function $G(r, \theta)$ under the above assumptions. We also use the familiar Landau symbol:\\
$f(n) = o(g(n))$ which means that for all $c > 0$ there exists some $k > 0$ such that $0 \le f(n) < cg(n)$ for all $n \ge k$. The value of $k$ must not depend on $n$, but may depend on $c$.
The first lemma can be proved by considering $\theta_0=0$. 
\begin{lemma}\label{HL1}
Let 
$$
g (\theta)- g (\theta_0) = o (|\theta-\theta_0|)
$$
where $0 < \alpha < 1$ and $\theta \to \theta_0$. Then
$$
 \frac{\delta G (r, \theta_0)}{\delta\theta_0} = o{(1-r)^{(1-\alpha)}}
$$
whenever  $r\to 1$.
\end{lemma}
The next lemma is a  well-known result and interested reader  can find a proof of it in \cite{Fatou}. 
\begin{lemma}\label{HL2}
Suppose $g(\theta)$ has a finite differential coefficient $g'(\theta_0)$ for $\theta=\theta_0$. Then 
$$
\frac{\delta G}{\delta \theta_0}\to g'(\theta_0)
$$
with $r\to 1$.
\end{lemma}
The next result is a special case of a general theorem proved by J. E. Littlewood and G. H. Hardy in \cite{HL}. 
\begin{lemma}\label{HL3}
Let $f(y)$ is a real or complex valued function of the real
variable $y$, possessing a $p^{th}$ differential coefficient $f^{(p)} (y)$ which is continuous in $(0, y_0]$. Let $\lambda\geq 0$ and that
$$
f (y) = o (y^{-\lambda})
$$
whenever $\lambda > 0$ and
$$
f (y) =A + o(1)
$$
whenever $\lambda = 0$. Also in either cases that
$$
f^{ (p)} (y) = O ( y^{-p-\lambda}).
$$
Then
$$
f^{(q)}(y)=o(y^{-q-\lambda})
$$
for $0 < q < p$.
\end{lemma}
Now by setting $e^{-y}=u,\ f(y)=\sum a_nu^n$ and then applying 
$$
a_0+a_1+\cdots+a_n=s_n=1+b^\rho+b^{2 \rho}+\cdots+b^{\nu \rho}
$$
for $b^\nu\leq n<b^{\nu+1}$, one can easily get:
\begin{lemma}\label{HL4}
Let us suppose $\rho > 0$ and that 
$f (y) = \sum b^{n\rho}e^{-b^ny}$.
Then
$$
 f (y) =O (y^{-\rho})
$$
as $y\to 0$.
\end{lemma}
The next result is also not difficult to prove.
\begin{lemma}\label{HL5}
 Let
$$
\sin b^n\pi x\to 0
$$
as $n\to 0$. Then
$x=\frac{p}{b^q}$ for some integers $p$ and $q$.
\end{lemma} 
\begin{remark}
It is clear from the above lemma that $\sin b^n\pi x=0$ for $n\geq q$.
\end{remark}
To state the next lemma, one needs the following notations which was introduced by G. H. Hardy and J. E. Littlewood in \cite{HL14}. The notation $f = \Omega(\phi)$
basically signifies the negation of $f = o (\phi)$, that is to say as asserting the existence of a constant $K$ such that $|f|>K\phi$ for some special sequence of values whose limit is that to which the variable is supposed to tend. The sequence that one can use to prove the following lemma is the values of $y$, that is $y=\frac{\rho}{b^m}$ for $m=1, 2, 3, \ldots$
\begin{lemma}\label{HL6}
Suppose that
$$
 f (y)= \sum b^{n\rho} e^{-b^ny}\sin b^n\pi x,
$$
where $y > 0$, and that
$$
x\ne \frac{p}{b^q}
$$
for any integral values of $p$ and $q$. Then
$$
f(y) =\Omega (y^{-\rho})
$$
for all sufficiently large values of $\rho$.  
\end{lemma}
We are now in a position to give outline of the proofs of Theorem \ref{HT1} and Theorem \ref{HT2}. We  give the proof for cosine series and then we provide the outline of the same for the sine series. 
We begin the proof with the following conditions:
\begin{equation}\label{wf7}
ab> 1
\end{equation}
and 
\begin{equation}\label{wf8}
x\ne \frac{p}{b^q}.
\end{equation}
Let us suppose
\begin{equation*}\label{wf9}
f(x) =\sum a_n \cos b^n\pi x=\sum a_n \cos b^n\theta = g(\theta)
\end{equation*} 
satisfy the condition
\begin{equation*}\label{wf10}
f (x + h) -f (x) = o (|h|^\xi).
\end{equation*}
That is,
\begin{equation}\label{wf11}
g (\theta + h) -f (\theta) = o (|h|^\xi)
\end{equation}
with 
\begin{equation*}\label{wf12}
\xi=\frac{\log (1/a)}{b}<1.
\end{equation*}
Then if 
$$
G (r, \theta) =\sum a_n r^{b^n} \cos b^n\theta=\sum a_n e^{-b^ny} \cos b^n\pi x,
$$
we have (using Lemma \ref{HL1}),
\begin{eqnarray*}
 F(y)=\frac{\delta G}{\theta}&=&-\sum (ab)^n e^{-b^ny} sin b^n \pi x\\
&=&-\sum b^{(1-\xi)n} e^{-b^ny} sin b^n \pi x\\
& =& o(y^{\xi -1})
\end{eqnarray*}
when $r\to 1,\ y\to 0$.
 
Again using Lemma \ref{HL4}, one has,
\begin{eqnarray*}
F^{(p)} (y)& =& ( - 1 )^{p+1}\sum (ab^{p+1})^n e^{-b^ny} \sin b^n \pi x\\
& =& O\bigg( \sum b^{(p+1-\xi)n}e^{-b^ny}\bigg)\\
& =& O(y^{\xi-1})
\end{eqnarray*}  
 for all positive values of $p$. It follows from Lemma \ref{HL3} that
$$
F^{(q)} (y) = o(y^{\xi-q-l})
$$
for $0 < q < p$, and thus for all positive values of $q$. But this contradicts the assertion of  Lemma \ref{HL6}, if $q$ is sufficiently large. Hence the conditions (\ref{wf11}) can not besatisfied.
The case in which
$$
 ab = 1,\ \ \xi = 1,
$$
may be treated in the same manner. The only difference is that one should use
Lemma \ref{HL2} instead of Lemma \ref{HL1}, and that the final conclusion is that:\\
$f (x)$ cannot possess a finite differential coefficient for any value of $x$ which is not of the form $\frac{p}{b^q}$.
 
This approach though fails in the case when $x = \frac{p}{b^q}$. These values of $x$ need to be treated differently. In this case,
$$
\cos \{b^n \pi (x + h)\} = \cos (b^{n-q} p\pi + b^n\pi h) = \pm \cos b^n\pi h
$$
for $n > q$. One takes negative signif both $b$ and $p$ are odd, and
positive sign otherwise. Therefore the properties of the function in the neighbourhood of such a value of $x$ are the same, for the present purpose, as those
of the function
$$
f (h) =\sum a^n \cos b^n\pi h
$$
when $h\to 0$. Thus
\begin{eqnarray*}
f (h) -f (0) &=&-2\sum a^n \sin^2\frac{1}{2}b^n\pi h\\
&=& - 2\sum (f_1 +f_2).
\end{eqnarray*}
Here,
$$
 f_1 = \sum_0^{\nu} a^n \sin^2\frac{1}{2}b^n\pi h\ \ \text{and}\ \ f_2 =\sum_{\nu+1}^\infty a^n sin^2\frac{1}{2} b^n \pi h.
$$
We now choose $\nu $ in such a way that
\begin{equation}\label{wf13}
 b^\nu |h|\leq  b^{\nu+1}|h|.
\end{equation}
Then
\begin{eqnarray*}
f_1 +f_2 > f_1 > \sum_0^\nu a^n (b^nh)^2 &=&h^2\frac{(ab^2)^{\nu+1}-1}{ab^2-1}\\
&>& Kh^2 (ab^2)^\nu > Ka^\nu > Kb^{-\xi\nu} > K|h|^\xi
\end{eqnarray*}
where the $K$'s are constants. Therefore,
$$
f (h) )-f (0)\ne o (| h|^\xi).
$$
This completes the proof when $ab > 1,\ \xi < 1$. In this case the graph of $f (h)$ has a cusp (pointing upwards) for $h = 0$, and that of Weierstrass's function has a cusp for $x = \frac{p}{b^q}$. On the other hand, if $ab = 1,\ \xi = 1$, then it is proved that
$$
{\over{\lim}}_{h\to 0+}\frac{f(h)-f(0)}{h}<0,
$$
and
$$
{\underline{\lim}}_{h\to 0-}\frac{f(h)-f(0)}{h}>0,
$$
so that $f (h)$ has certainly no finite differential coefficient for $h = 0$, nor the 
Weierstrass's function has for $x = \frac{p}{b^q}$.
This completes the proof of Theorem \ref{HT1} and Theorem \ref{HT2} in so far as they relate to the cosine series and are of a negative character. Only part remains is to show that, when $\xi < 1$, Weierstrass's function satisfies the condition
$$
 f(x + h) -f(x) = O (|h|^\xi)
$$
for all values of $x$.  One starts with the left hand side:
\begin{eqnarray*}
f(x+h)-f(x)& =& -2\sum a^n\sin \{b^n\pi (x+ h)\} \sin\frac{1}{2}b^n \pi h\\
 &=& O\bigg(\sum a^n| \sin\frac{1}{2} b^n \pi h\bigg).
\end{eqnarray*}
Again choose $\nu $ as in (\ref{wf13}) and then we have
\begin{eqnarray*}
f(x+h)-f(x)& =& O(|h| \sum_0^{\nu} a^nb^n+\sum_{\nu+1}^\infty a^n)\\
 &=& O(a^\nu b^\nu |h|+a^\nu)\\
 &=&O(a^\nu)\\
 &=&O(|h|^\xi).
\end{eqnarray*}
Hence the condition is satisfied, and in fact it holds uniformly in $x$. It is observed that the above argument fails when $ab = 1,\ \xi = 1$. In this case though one can
only say that
 $$
 f(x + h) -f(x) = O (\nu |h| I + a^\nu) = O\bigg(|h| \log\frac{1}{|h|}\bigg).
 $$
It is also be observed that the argument of this paragraph applies to
the cosine series as well as to the sine series. This is indeed independent of the restriction that $b$ is an integer.
 
The proof of Theorem \ref{HT1} and Theorem \ref{HT2} are now complete so far as the
cosine series is concerned. The corresponding proof for the sine series differs
only in detail. The subsidiary results required are the same except that Lemma \ref{HL5} is being replaced by the following one. 
\begin{lemma}\label{HL7}
If
$$
\cos b^n\pi x\to 0,
$$
then $b$ must be odd and
$$
x=\frac{p+\frac{1}{2}}{b^q};
$$
so that $\cos b^n\pi x = 0$ from a particular value of $n$ onwards.
Also the corresponding changes must be made in Lemma \ref{HL6}.
\end{lemma}
If the value of $x$ is not exceptional (i.e. one of those as is specified in Lemma \ref{HL7}), one can repeat the arguments that was used in the case when (\ref{wf7}) and (\ref{wf8}) hold. Thus it is only necessary to discuss the exceptional values, which can exist only when $b$ is odd. In this case we have,
\begin{eqnarray*}
\sin \{b^n\pi(x+h)\} &=&\sin(b^{n-q}p\pi+ \frac{1}{2}b^{n-q}\pi+b^n\pi h)\\
&=&\pm \sin (\frac{1}{2}b^{n-q}\pi + b^n\pi h),
\end{eqnarray*}
for $n > q$, the sign has to be fixed as in the case when $x=\frac{p}{b^q}$. The last function is numerically equal to $\cos b^n\pi h$. It always has the same sign as $\cos b^n \pi h$, or always the opposite sign, if $b = 4k + 1$. While whenever $b$ is of the form $4k+3$, the corresponding signs agree and differ alternatively. Therefore we are reduced to either to discuss the function 
$$
f (h) =\sum a^n \cos b^n\pi h
$$ 
near $h = 0$, or to that of
$$
f (h) =\sum (-a)^n \cos b^n\pi h.
$$
The need is to show that
$$
f ( h )-f (0)\ne o(|h|^\xi)
$$
if $\xi < 1$, and that $f (h)$ has no finite differential coefficient for $h = 0$, if $\xi= 1$.

To do this let us consider the special sequence of values
 $$
 h = \frac{2}{b^\nu}\qquad (\nu = 1,\ 2,\ 3,\cdots).
 $$
Then we have
\begin{eqnarray*}
f(h)-f(0) &=&-2\sum_{0}^{\nu-1}(-a)^n \sin^2\frac{1}{2} b^n \pi h\\
 &=&(-1)^\nu 2a^{\nu-1}\sum_1^{\nu}\left(-\frac{1}{a}\right)^n \sin^2 \frac{\pi}{b^n}.
 \end{eqnarray*}
Now 
$$
\sum_1^\nu \left(-\frac{1}{a}\right)^n \sin^2\frac{\pi}{b^n}\to \sum_1^\infty \left(-\frac{1}{a}\right)^n \sin^2\frac{\pi}{b^n}=S  \hspace*{1mm}{\mbox{(say)}}.
$$ 
Now as $S$ is the sum of an alternating series of decreasing terms, it is positive.
Again, we have
$$
a^\nu = b^{-\xi\nu}=\left(\frac{1}{2}h\right)^\xi.
$$
Thus 
$$
|f (h) - f (0)| > c h^\xi,
$$ 
for some constant $c$ and is alternately positive and negative. This completes the proof of Theorem \ref{HT1} and Theorem \ref{HT2}.

Now the time is to comeback to Remark \ref{HR1}, that is the question remains whether an equally comprehensive result holds for infinite differential coefficients. The result that includes the Remark \ref{HR1}, shows that the answer to this question is negative.
\begin{theorem}\label{HT3}
If
$$ 
ab\geq 1\ \ \text{and}\ \ a(b + a) < 2
$$
then the sine series has the differential coefficient $+\infty$ for $x = 0$. If $b = 4k + 1$ and $x=\frac{1}{2}$,
then the same is true of the cosine series .
\end{theorem}
It is enough to prove the first statement. The second one then follows by the transformation $x =\frac{1}{2} + y$.

We have,
\begin{eqnarray*}
 \frac{f (h)-f(0)}{h}&=&\frac{1}{h}\sum a^n \sin b^n\pi h\nonumber\\
 & =& \frac{1}{h}\sum_0^{\nu-1} a^n \sin b^n\pi h + \frac{1}{h}\sum_\nu^\infty a^n \sin b^n \pi h\nonumber\\
& =& f_1 + f_2 \hspace*{1mm} \mbox{(say)}.
\end{eqnarray*}
Here $\nu $ has to be chosen so that 
\begin{equation}\label{wf14}
b^{\nu-1}|h|\leq\frac{1}{2}<b^\nu|h|.
\end{equation}
We first suppose that $ab > 1$ and then,
\begin{equation}\label{wf15}
f_1 > 2\sum_0^{\nu-1}(ab)^n=2\frac{(ab)^\nu-1}{ab-1}
\end{equation}
and
 \begin{equation}\label{wf16}
|f_2| <\frac{1}{|h|} \sum_\nu^{\infty}a^n=\frac{a^\nu}{(1-a)|h|}.
\end{equation}
Now it is clear that
$$
a(b + 1) < 2,\qquad 1 - a > ab - 1
$$
and thus
\begin{equation}\label{wf17}
\frac{1-a}{ab-1}=1+\delta
\end{equation}
where $\delta > 0$. Without loss of generality, one can assume $h$ is so small (or $\nu$ is so large) so that
\begin{equation}\label{wf18} 
\frac{(ab)^\nu-1}{(ab)^\nu}>\frac{2+\delta}{2(1+\delta)}.
\end{equation}
Then from (\ref{wf14})-(\ref{wf18}) it follows that           
$$
\frac{|f_2|}{f_1}<\frac{(ab)^\nu(ab-1)}{\{(ab)^\nu-1\}(1-a)}<\frac{1}{1+\frac{1}{2}+\delta}.
$$
Thus 
$$
f_1+ f_2 > c_1 f_1 \hspace*{1mm} \mbox{or} \hspace*{1mm} c_2 (ab)^\nu,
$$
for some constants $c_i, i=1,2$.
Thus we have
\begin{equation}\label{wf19}
\frac{f(h)-f(0)}{h}\to +\infty
\end{equation} 
as $h\to 0$.

Next, if $ab = 1$, then $|f_2| < k$, a constant, and that
$$
 f_1 > 2\nu\to +\infty.
$$
Hence (\ref{wf19}) remains true in this case too.

A number $\alpha (b)$ exists when $b$ is given and it is simply the least number such that the condition
$$
ab > \alpha(b).
$$
This debars the existence of a differential coefficient whether finite or infinite. At present all that one 
can say about $\alpha (b)$is that 
$$
\frac{2}{b+1}\leq \alpha(b)\leq \frac{1+\frac{3}{2}\pi}{b+\frac{3}{2}\pi}.
$$
\subsection{$b$ is not an integer.}
One needs to discuss everything those are stated in previous sub-section with
$b$ is no more an integer and thus the series are no longer Fourier series, and one can no longer has the luxary to employ Poisson's integral associated to $G(r,\theta)$. 

The job is naturally to construct a new formula to replace the Poisson's one. Once it is has been achieved,  further modifications of the argument are needed. This is because of the lack of any simple result corresponding to Lemma \ref{HL5} and Lemma \ref{HL7}, and the difficulty of determining precisely the exceptional values of $x$ for which $\sin b^n\pi x\to 0$ or $\cos b^n \pi x\to 0$. The beauty of the argument is that, however, it will be found that no fundamental change in the method is necessary. Also that the additional analysis required is not complicated.

Let $b$ is any number greater than $1$. Also
$$
s = \sigma + it,
$$
(usual in the theory of Dirichlet  series), 
and that
\begin{equation*}\label{wf20}
f(s) =\sum_1^\infty a^n e^{-b^ns} = G(\sigma, t) + iH(\sigma, t), \qquad (\sigma\geq 0)
\end{equation*}
with the condition that
\begin{equation*}\label{wf21}
 G (0, t) = g(t).
\end{equation*}
Then one can show that:
 \begin{lemma}\label{HL8}
 Let $\sigma > 0$. Then
$$ 
G (\sigma, t) =\frac{1}{\pi}\int_{-\infty}^\infty \frac{\sigma g(u)}{\sigma^2+(u-t)^2} du.
$$
\end{lemma}
First let us set $ab>1$. In this case one uses  the following ones instead of Lemma \ref{HL1}:
\begin{lemma}\label{HL9}
If $$g (t) - g (t_0) = o (|t - t_0|^\alpha),$$
where $0 < \alpha < 1$, when $t\to t_0$. Then
$$
\frac{\delta G(\sigma, t_0)}{\delta t_0}=o(\sigma^{\alpha-1}),
$$
whenever $\sigma\to 0$.
\end{lemma}
\begin{lemma}\label{HL10}
$$
\frac{\delta G(\sigma, t_0)}{\delta \sigma}=o(\sigma^{\alpha-1})
$$
under the same conditions as in the previous Lemma.
\end{lemma}
The proofs of these lemmas are very similar, and the first is similar
to that of Lemma \ref{HL1}. One can consult \cite{H} for detail of the proofs. 
 
We now discuss the exceptional values of $t$. Suppose that
$$
 g(t + h) -g(t) =o(|h|^\xi).
$$
Then, by Lemma \ref{HL9} and Lemma \ref{HL10}, we have
$$
\frac{\delta G}{\delta t}=-\sum (ab)^n e^{-b^n \sigma}\sin b^n t = o(\sigma^{\xi-1} )
$$
and
$$
\frac{\delta G}{\delta \sigma}=-\sum (ab)^n e^{-b^n \sigma} \cos b^n t = o(\sigma^{\xi-1} ).
$$
Therefore we have
$$
f (y) = \sum (ab)^n e^{-b^n (\sigma+it)}= o(\sigma^{\xi-1} ).
$$
One can now obtain a contradiction by employing the same argument as used earlier when we consider (\ref{wf7}) and (\ref{wf8}). It is only necessary to observe that Lemma \ref{HL3} holds for complex as well as for real functions of a real variable. Also instead of Lemma \ref{HL7}, one has to use the following proposition.
\begin{proposition}\label{HP1}
If
$$
f (y)=\sum b^{n\rho}e^{-b^n(\sigma+it)}, \qquad (\sigma>0),
$$
then
$$
f(y) =\Omega (\sigma^{-\rho})
$$
for all sufficiently large values of $\rho$.
\end{proposition}
There is no longer any
question of exceptional values of $t$ as $|e^{-b^nit}| = 1$.
 
Next we treat when $ab=1$. In this case instead of Lemma \ref{HL9}, one uses the following result (which corresponds to Lemma \ref{HL2}).
\begin{lemma}\label{HL11}
Let $g (t)$ possesses a finite differential coefficient $g' (t_0)$ for
 $t = t_0$. Then
$$
\frac{\delta G(\sigma, t_0)}{\delta t_0}\to g'(t_0)
$$
when $\sigma \to 0$.
\end{lemma}
The proof of this lemma is no more difficult. One needs to keep in mind though that it is not necessarily true that
$$
\frac{\delta G(\sigma, t_0)}{\delta t_0}
$$
tends to a limit. Thus it is necessary to follow a
slightly different argument from that of when we treated the exceptional values of $t$.
\begin{lemma}\label{HL12}
Under the same conditions as those of Lemma \ref{HL11}, we have
$$
\frac{\delta^2 G(\sigma, t_0)}{\delta t_0^2}= o\bigg(\frac{1}{\sigma}\bigg).
$$
\end{lemma}
Suppose  that $g (t)$ has a finite differential coefficient $g' (t)$, and write
$$ 
f(\sigma) =\frac{\delta G}{\delta t}=-\sum e^{-b^n\sigma} \sin b^n t.
$$
Then, by Lemma \ref{HL11}, we have
$$
f(\sigma) =g'(t)+o(1)
$$
when $\sigma\to 0$. But by Lemma \ref{HL4} we have,
$$ 
f''(\sigma) = -\sum b^{2n} e^{-b^n\sigma} sin b^n t = O\big(\frac{1}{\sigma^2}\big).
$$
Therefore, by Lemma \ref{HL3},
\begin{equation}\label{wf22}
f'(\sigma) =\sum b^n e^{-b^n\sigma} \sin b^n t =o\left(\frac{1}{\sigma}\right).
\end{equation}
On the other hand, by Lemma \ref{HL12}, we have
\begin{equation}\label{wf23}
\frac{\delta^2 G}{\delta t^2} =-\sum b^n e^{-b^n\sigma} cos b^n t = o\left(\frac{1}{\sigma}\right).
\end{equation} 
Now from (\ref{wf22}) and (\ref{wf23}) it follows that
$$
 F (\sigma) = \sum b^n e^{-b^n(\sigma+it)}=o\big(\frac{1}{\sigma}\big).
 $$
Also, by Lemma \ref{HL4}, we have
$$ 
F^{(p)} (\sigma) =(-1)^p \sum b^{(p+1)n} e^{-b^n(\sigma+it)}=o\big(\frac{1}{\sigma^{p+1}}\big),
$$
for all values of $p$. Thus it follows that the $O$ can be replaced
by $o$; and this leads to a contradiction as before.
 
Finally the following remark completes the proof.
 \begin{remark}The above argument,
 has been stated in terms of Weierstrass's cosine
 series. The same arguments apply to the sine series, as there are now no
 \enquote{ exceptional values}. It was only the existence of such values which
 differentiated the two cases in second subsection. The positive statement in Theorem \ref{HT2} has already been proved, applying to all values of $b$. 
 \end{remark}
\section{Some more non-differentiable functions}
In this section we discuss some more non-differentiable functions which are available in \cite{H}.
\subsection{A function which doesn't satisfy a Lipschitz condition of any order.}
It is interesting to given an example of an absolutely convergent Fourier series whose sum does not satisfy any condition of the following type:
$$ 
f(x + h) - f(x) = O(|h|^\alpha), \qquad ( \alpha > 0)
$$
for any value of $x$. 
An interesting example of such a function is: 
$$
f(x)=\sum\frac{\cos b^n \pi x}{n^2}.
$$
It is in fact easy to prove, by the methods used in the previous section, that
$$ 
f(x + h) -f(x)\ne o\bigg(\frac{1}{| \log|h||}\bigg)^2.
$$
However, a somewhat less simpler example may be
found by simply combining remarks made by G. Faber and G. Landsberg. In \cite{Faber}, G. Faber defined
\begin{equation}\label{wf24}
F(x) = \sum 10^{-n}\phi(2^{n!x}),
\end{equation} 
where 
$$\phi(x) = 
\begin{cases}
x, & for   \hspace*{2mm}0\le x\le 1/2; \cr
 1-x, & for \hspace*{2mm}1/2 \le x\le 1. \cr
  \end{cases}
$$
He showed that
$$
F(x + h) - F(x)\ne O\bigg(\frac{1}{|\log|h||}\bigg).
$$
On the other hand, G. Landsberg \cite{Landberg} used the expansion of a function,
which is in a Fourier series equivalent to $\phi(x)$. In fact, 
$$
\phi(x)=\frac{1}{4}-\frac{2}{\pi^2}\sum \frac{\cos 2\nu \pi x}{\nu^2} \qquad (\nu=1, 3, 5, \cdots).
$$
If we substitute this expansion in (\ref{wf24}), we obtain an expansion of $F (x)$
as an absolutely convergent Fourier series, and thus is  an example of the kind we are looking for.
 \subsection{A theorem of S. Bernstein.} It is natural to
suggest following theorem of S. Bernstein \cite{Bernstein} in this connection. This can be proved similarly as is being done in the previous section. 
 \begin{theorem}\label{HT4}
If $f (x)$ satisfies a Lipschitz condition of order $\alpha\ (>2)$ in
$(0,\ 1)$, i.e. if
$$
|f(x + h) -f(x)| < K|h|,
$$
where $K$ is an absolute constant. Then the Fourier series of $f (x)$ is absolutely convergent. Also $\frac{1}{2}$ is the least number which has this property.
\end{theorem} 
\begin{proof}
We assume that $2\pi x = \theta$ and that
$$
f(x) = g(0) = \frac{1}{2}a_0 + \sum (a_n \cos n\theta + b_n \sin n\theta).
$$
Also let
$$
G(r, \theta) = \frac{1}{2}a_0 + \sum r^n(a_n \cos n\theta + b_n \sin n\theta)\hspace*{1mm}\mbox{if} \quad r<1,
$$
and 
$$G(1, \theta) = g(\theta).$$ Then $G(r, \theta)$ is continuous for $$0\leq r\leq  1,\ 0\leq \theta\leq  2\pi.$$
 
It follows from a simple modification (i.e., O in place of o) of Lemma \ref{HL1}, that 
$$
\frac{\delta G}{\delta\theta}= -\sum nr^{n-1}(a_n \sin n\theta -b_n \cos n\theta) = O\{(1 -r)^{\alpha-1}\},
$$
 uniformly in $\theta$. Squaring, and integrating from $\theta  = 0$ to $\theta = 2\pi$, one can obtain
 $$
 \sum n^2 r^{2n} (| a_n |^2 + |b_n |^2) = O(1 -r)^{2\alpha-1}.
 $$
 Hence, by putting $r = 1 - (1/\nu)$, one can obtain
$$
 \sum_1^\nu  n^2( {|a_n|}^2+ {|b_n|}^2) = O(\nu^{2-2\alpha}),
$$
and so, by Schwarz's inequality,
$$
\sum_1^\nu n (|a_n| +|b_n|) =O(\nu^{\frac{3}{2}\alpha}).
$$
Thus it is easy to deduce that the series
$$
\sum_1^\nu n^\beta (|a_n| + |b_n|)
$$
is convergent if $\beta <\alpha-\frac{1}{2}.$

This establishes the first part of Bernstein's Theorem (indeed
more!). The tsecond part is shown by the following example:
$$
 g(\theta) = \sum n^{-b} \cos (n^a + n\theta),
 $$
where $0 < a < 1,\ 0 < b < 1$. In this case $G(r, \theta) $ is the real part of
$$
 F(z) = F(re^{i\theta}) = \sum n^{-b} e^{in^a} z^n.
 $$
This function is continuous (see \cite{H13}) for
$|z|\geq 1$ if
$$
\frac{1}{2} a + b > 1;
$$
and it is not difficult to go further, and to show that $g (\theta)$ satisfies a Lipschitz condition of order $\frac{1}{2}a + b - 1$. 
 
Now let $\alpha$ be any number less than $\frac{1}{2}$. Then one can choose numbers $a$ and $b$, each less than $1$ in such a way that
$$
\frac{1}{2}a + b-1 > \alpha.
$$
Then the function $g (\theta)$ satisfies a Lipschitz condition of order greater than $\alpha$, but its Fourier series is not absolutely convergent. 
\end{proof}
\section{Riemann's non-differentiable function revisited}
Riemann is reported to have stated \cite{BR, H}, but never proved, that the continuous function,
$$
f(x) = \sum_{n=1}^{\infty} \frac{\sin n^2x}{n^2}
$$
is nowhere differentiable. In 1872, K. Weierstrass \cite{Wei} tried to prove this assertion  but couldn't and instead constructed his own example of a continuous nowhere differentiable function defined by (\ref{wf1}) along with the conditions (\ref{wf2}) and (\ref{wf3}). 
J. P. Kahane \cite{Ka} renewed the interest in this classical problem in connection with lacunary series, and refers  to
K. Weierstrass \cite{Wei}. 

Riemann's assertion was partially confirmed by G. H. Hardy \cite{H}, who proved that the above function $f(x)$ has no finite derivative at any point $\xi \pi$,
where $\xi$ is:
\begin{itemize}
\item[(i)] irrational;
\item[(ii)] rational of the form $\frac{2A}{4B+1}$, where $A$ and $B$ are integers;
\item[(iii)] rational of the form $\frac{2A+1}{2(2B+1)}$.
\end{itemize}
We provide an outline of Hardy's method in this case. Suppose that Riemann's function is differentiable for certain values of $x$, then  by Lemma \ref{HL2}, 
$$
\sum r^{n^2}\cos n^2\pi x = A + o (1),
$$
where $A$ is a constant, as $r\to 1$.
However,
$$
\sum r^{n^2}\cos n^2 \pi x = \Omega \{(1 -r)^{- \frac{1}{4}}\}
$$
 if $x$ is irrational, and
$$
\sum r^{n^2}\cos n^2\pi x =\Omega\{(1-r)^{- \frac{1}{2}}\}
$$
if $x$ is a rational of the form $\frac{2\lambda + 1 }{2 \mu}$ or $\frac{2\lambda}{4 \mu + 1}$. Therefore Riemann's function is certainly not differentiable for any irrational (and some rational) values of $x$. It is easy, by using Lemma \ref{HL1}, instead of Lemma \ref{HL2}, to show that Riemann's function cannot satisfy the condition
$$ 
f(x + h) -f( x) = o(| h|^\frac{3}{4})
$$
for any irrational values of $x$. In this context Hardy \cite{H} proved the following theorem:
\begin{theorem}\label{HT5}
None of the functions
$$
f_{c,\alpha}(x)=\sum \frac{\cos n^2 \pi x}{n^\alpha}
$$ 
and
$$
f_{s, \alpha}(x)=\sum \frac{\sin n^2 \pi x}{n^\alpha}
$$
where $\alpha <\frac{5}{2}$, is differentiable for any irrational value of $x$.
\end{theorem}
\begin{proof}
Suppose that $f_{s, \alpha}$ is differentiable and consequently  Lemma \ref{HL12} would imply,
$$ 
\sum n^{2-\alpha}r^{n^2} \cos n^2 \pi x = A +o(1),
$$
or
$$
 f(y) = \sum n^{2-\alpha} e^{-n^2y} \cos n^2 \pi x = A +o(1).
$$
But,
\begin{eqnarray*}
f^{(p)} (y)&=& ( - 1 )^p \sum n^{ 2p+2-\alpha} e^{-n^2y} \cos n^2 \pi x\\
 &=& O\big(\sum n^{ 2p+2-\alpha} e^{-n^2y} \big) = O (y^{-p-\frac{3}{2}+\frac{\alpha}{2}}).
\end{eqnarray*}
 Hence by the theorem of Hardy and Littlewood \cite{HL14}, we have
 $$
 f^{(q)} (y)=o \left(y^{-\frac{q}{p}(p+\frac{3}{2}-\frac{\alpha}{2}}\right).
 $$
 Here $0 < q < p$, and in particular
 \begin{equation}\label{wf25}
 f' (y)= o \left(y^{-1-\frac{3}{2p}+\frac{\alpha}{2p}}\right).
 \end{equation}
 Again, it is easy to prove that
 \begin{equation}\label{wf26}
 f'(y)= -\sum n^{4-\alpha} e^{-n^y} \cos 2\pi x=\Omega(y^{-\frac{9}{4}+\frac{\alpha}{2}}).
\end{equation}  
From (\ref{wf25}) and (\ref{wf26}) it follows that
$$
1+\frac{3}{2p}-\frac{\alpha}{2p}>\frac{9}{4}-\frac{\alpha}{2}.
$$
But this is not possible if $\alpha  < \frac{5}{2}$ and $p$ is sufficiently large.
It is clear that the series
$f_{c, \beta}$ and $f_{s, \beta}$ with $0 < \beta < \frac{1}{2}$ are not Fourier's series. For if the first one is a Fourier's series, then the sum of the integrated series $f_{s, 2+\beta}$ would be a function of limited total fluctuation, and would therefore be differentiable almost everywhere. 
 
 It is easy to prove directly that the function
 $f_{s, \alpha}$, where $2 < \alpha < \frac{5}{2}$, has the differential coefficient $+ \infty$ for $x = 0$. A similar direct method could no doubt be applied to an everywhere dense set of rational values of $x$.
 \end{proof}
 In 1970, J. Gerver \cite{G1} proved that Riemann's assertion is false, by proving the following result. 
 \begin{theorem}\label{GT1}
The derivative of the following function
$$
f(x) = \sum_{n=1}^{\infty} \frac{\sin n^2x}{n^2}
$$ 
exists and is equal to $- \frac{1}{2}$ at any point $\frac{(2A+1)\pi}{2B+1}$, where $A$ and $B$ are integers. 
 \end{theorem}
In the same paper, J. Gerver \cite{G1} extended G. H. Hardy's results \cite{H} by proving the following:
 \begin{theorem}\label{GT2}
The derivative of the Riemann functions does not exist at any point $\frac{(2A+1)\pi}{2^N}$, where $N$ is an integer $\geq 1$ and $A$ is any integer.
\end{theorem}
One can consult \cite{G1} for detailed proof of Theorem \ref{GT1} and Theorem \ref{GT2}. 
 
In 1971, J. Gerver further proved some results concerning the non-differentiability of Riemann's function. More precisely, he proved the following:
 \begin{theorem}\label{G3}
The function 
$$
f(x) = \sum_{n=1}^{\infty} \frac{\sin n^2x}{n^2}
$$  is not differentiable at any point $\frac{2A\pi}{2B+1}$ or $\frac{\pi(2A+1)}{2B}$, where $A$ and $B$ are integers.
\end{theorem}
This result together with Hardy's result \cite{H} that the function is not differentiable at any irrational multiple of $\pi$, completely solves the problem of differentiability.
 
In 1972, A. Smith \cite{Sm} extended the above results to the remaining cases. He also discussed the existence of finite left-hand and right-hand derivatives at certain rationals, and proved that these derivatives exist at all rationals if the
values $\pm \infty$ were allowed. He gave completely elementary and fairly short
proof of all the above assertions. J. Gerver's proof was extremely long and G. H. Hardy obtained his results indirectly. A. Smith worked with the following function
$$
g(x)=x+2 \sum_{n=1}^\infty \frac{\sin n^2 \pi x}{\pi n^2}, 
$$
so that, one can verify that $g'(x)$ exists and is zero whenever $x$ is of the form $\frac{2A+1}{2B+1}$ for some integers $A$ and $B$.

The following lemmas are required to obtain expansions for $g(x)$ about a rational point $x$, which using properties of Gaussian sums reveal the properties of the derivatives.
\begin{lemma}\label{SL1}
Let $\phi$ be a continuous function in $L_1(-\infty, \infty)$. Suppose that the series for $Q(\alpha)$ (defined below) converges uniformly
 in every finite $\alpha$ interval, for each fixed $h >0$. Let
 $$ 
 \hat{\phi}(y)=\int_{-\infty}^\infty e^{-2\pi i xy} \phi (x)dx
 $$
 and assume that $|y|^\beta |\hat{\phi}(y)|$ is bounded for some fixed $\beta >1$.
 Then for any real constant $\alpha$, as $h\to 0+$,
$$
Q(\alpha)= \sum_{k=-\infty}^\infty h\phi(hk + h\alpha) =\hat{\phi}(0) + O(h^\beta).
$$
\end{lemma}
\begin{proof}
The conditions on $\phi$ allows one to apply the Poisson summation formula to 
$$
\sum_{k=-\infty}^\infty h\phi(hk + h\alpha)
$$ 
to obtain
$$
\sum_{k=-\infty}^\infty h\phi(hk + h\alpha)=\sum_{k=-\infty}^\infty e^{2\pi i k \alpha}\hat{\phi}\bigg(\frac{k}{n}\bigg)
$$
provided this series converges absolutely.
The condition on $\hat{\phi}$ gives, for $k\ne 0$,
$$
e^{2\pi i k \alpha}\hat{\phi}\bigg(\frac{k}{n}\bigg)=O\bigg(\frac{h^\beta}{|k|^\beta}\bigg)
$$
 which shows that the above sum, leaving out the $k=0$ term, converges
 absolutely and is $O(h^\beta)$. Thus
$$
\sum_{k=-\infty}^\infty e^{2\pi i k \alpha}\hat{\phi}\bigg(\frac{k}{n}\bigg)=\hat{\phi}(0)+O(h^\beta).
$$
\end{proof}
\begin{lemma}\label{SL2}
Let
$$
\phi_1 (x) = 
 \begin{cases}
 \frac{\sin \pi x}{\pi x},\qquad x\ne 0,\\
1, \qquad x= 0,
\end{cases}
$$
$$
\phi_2 (x) =
 \begin{cases}
 \frac{1-\cos \pi x}{\pi x},\qquad x\ne 0,\\
0, \qquad x= 0.
\end{cases}
$$
 Then Lemma \ref{SL1} with $\beta=2$ applies to the functions $\psi_i(x)=\psi_i(x^2),\ i= 1, 2$,
 and
 $$\sum_{k=-\infty}^\infty h\psi(hk + h\alpha) = 2^{1/2} + O(h^2),\ i = 1,2.$$ \end{lemma}
The following lemma is straight forward.
\begin{lemma}\label{SL3}
Assume that $ x=\frac{r}{s}$ and that $\ (r, s)=1$. Let us define
$$ 
G(x) =\sum_{t=0}^{s-1} e^{i\pi t^2 x} = C(x) + iS(x)\equiv\sum_{t=0}^{s-1} \cos \pi t^2 x + i\sum_{t=0}^{s-1}\sin \pi t^2 x;
$$ 
then\\
(a) when $r\equiv\pmod 2,$
$$
G(x) =
\begin{cases}
(\frac{r}{2s})s^{1/2}=1,\ \ s \equiv 1\pmod 4,\\
i(\frac{r}{2s})s^{1/2}=i, \ \ s\equiv 3 \pmod 4;
\end{cases}
$$ 
(b) when $s\equiv\pmod 2,$
$$
G(x) =
\begin{cases}
(\frac{r}{2s})\sqrt{\frac{s}{2}}(1+i),\ \ r \equiv 1\pmod 4,\\
(\frac{r}{2s})\sqrt{\frac{s}{2}}(1-i), \ \ r\equiv 3 \pmod 4;
\end{cases}
$$ 
where $(\frac{a}{b})$ denotes the Jacobi symbol;
\\
(c) when $rs=0\pmod 2$,
$$
|G(x)|=s^{1/2}.
$$
\end{lemma}
We are now in a position to discuss the derivative of $g(x)$ at rational and at some other points.
\subsection{The derivative at rational points}
We begin with the following assumptions:
$$  
x=\frac{r}{s},\ (r,\ s)= 1,\ rs\equiv 0\pmod 2.
$$
We have
\begin{eqnarray*}
g(x + h^2) + g(x - h^2) &=& 2x + 4 \sum_{n=1}^{\infty}\frac{\sin \pi n^2 x}{\pi n^2} \cos \pi n^2 h^2\\
&=& 2g(x) - 2h^2\sum_{n=-\infty}^\infty \sin\pi n^2x \psi_2(ph).
\end{eqnarray*}
Let us write $n=ks+t$ with $0\leq t\leq s-1$. Note that $\sin\pi(ks+t)^2x=\sin \pi t^2x$, since
 $rs\equiv 0 \pmod 2$. Then
\begin{eqnarray*}
g(x + h^2) + g(x - h^2) &=& 2g(x) - 2h^2\sum_{t=0}^{s-1}\sum_{k=-\infty}^\infty \sin\pi t^2x \psi_2(khs+ht)\\
&=& 2g(x) - 2\frac{h}{s}\sum_{t=0}^{s-1} \sin\pi t^2x \{2^{1/2}+O(h^2)\}\\
&=& 2g(x) - 2^{3/2}S(x)\frac{h}{s}+O(h^3).
\end{eqnarray*}
Note that Lemma \ref{SL2} is used in the penultimate line.\\ 
Similarly, we have
$$
g(x + h^2) -g(x - h^2) = 2^{3/2}C(x)\frac{h}{s}+O(h^3).
$$
Adding and subtracting these two equations, we obtain
\begin{equation}\label{SE1}
g(x \pm h^2) = g(x) - 2^{1}{2}\{S(x)\mp  C(x)\}\frac{h}{s} + O(h^3)
\end{equation}
We now assume that $rs\equiv 1 \pmod 2$. One can easily verify the relation 
$$
g(x)= 1 + \frac{1}{2}g(4x)-g(x+ 1)
$$
which is then used in (\ref{SE1}) to deduce that
\begin{equation}\label{SE2}
 g(x \pm h^2) = g(x) -2^{1/2}\{S(4x)- S(x + 1) \mp [C(4x) - C(x + 1)]\}\frac{h}{s} + O(h^3).
\end{equation} 
The properties of Jacobi symbols provide
$$
 \bigg(\frac{2r}{s}\bigg)=\bigg(\frac{2r + 2s}{s}\bigg)=\bigg(\frac{4((r + s)/2)}{s}\bigg)=\bigg(\frac{(r + s)/2}{s}\bigg),
$$
since $4$ is the square of the prime $2$ and $s\equiv 1\pmod 2$.
This immediately simplifies (\ref{SE2}) to
$$
g(x\pm h^ 2) =g(x) + O(h^3).
$$
Thus when $r\equiv s\equiv 1 \pmod 2$ we see that  $g'(x)$ exists and is $0$, since the right-hand derivative
$$
 g'_{+}(x)=\lim_{h^2\to \infty}\frac{g(x + h^2) - g(x)}{h^2}
 $$
and the left-hand derivative
$$
 g'_{-}(x)=\lim_{h^2\to \infty}\frac{g(x) - g(x-h^2)}{h^2}
 $$
both exist and are $0$. In this case, it follows that the symmetric derivative
 $$
 g'_{0}(x)=\lim_{h^2\to \infty}\frac{g(x + h^2) - g(x-h^2)}{2h^2}
 $$
also exists and is $0$. When $rs\equiv 0 \pmod 2$, the relation (\ref{SE1}) shows that $g'(x)$ is finite if and only if $G(x)=0$. However, by Lemma \ref{SL3}, $G(x)$ is not $0$. Hence $g'(x)$ is not finite when $rs$ even. One can easily verify that $g_{+}(x)$ when $r\equiv 1\pmod 4,\ g'_{-} (x)$ when $r\equiv 3 \pmod 4$, and $g_{0} (x)$ when $s\equiv 3 \pmod 4$ are all $0$, but in other cases these derivatives are infinite.
\subsection{Derivatives at other points}
At negative rationals the results of the preceding section carry over, since $g$ is an odd function.
 
We now assume that $x$ is irrational, which without loss generality we take to be positive. 
 Let $\{q_k\}$ be a strictly increasing sequence of positive integers, and let
$p_k$ be the least integer such that $x_k=\frac{2p_k}{4q_k+1}>x$. Then 
$x_k-x<\frac{2}{4q_k + 1}$ and $x_k\to x$ as $k\to \infty$. From (\ref{SE1}) and condition (a) of Lemma \ref{SL3}, we have
$$
\bigg| \frac{g(x)-g(x_k)}{x-x_k}\bigg|= \bigg\{\frac{1}{2}(4q_k + 1)(x_k - x)\bigg\}^{-1/2} + O((x_k-x)^{1/2} ).
$$
Therefore,
$$
\lim_{k\to \infty }\inf \bigg| \frac{g(x)-g(x_k)}{x-x_k}\bigg|\geq 1.
$$
Let $y_k=x_k+ \frac{1}{4q_k+ 1} = \frac{2p_k+ 1}{4q_k+ 1}$. Then $y_k\to x$ as $k\to \infty$ and 
 $$
 \lim_{k\to \infty } \frac{g(x)-g(y_k)}{x-y_k}=0.
 $$
From these two equations, we obtain Hardy's result that $g$ does not have a finite or infinite derivative at the irrational point $x$.

In 1981, S. Itatsu \cite{It} gave a short proof of the differentiability
as well as a finer estimate of the function 
$$
f(x)=\sum_{n=1}^\infty\frac{\sin n^2x}{n^2}
$$ 
at points of rational multiple of $\pi$. Namely, he proved the following result. 
\begin{theorem}{\label{IT1}}
The function
$$
F(x)=\sum_{n=1}^\infty \frac{e^{in^2\pi x}}{in^2\pi}
$$
have the following behaviour near $x=\frac{q}{p}$, where $p$ is a positive integer
and $q$ is an integer such that $\frac{q}{p}$ is an irreducible fraction,
$$
F(x+h)-F(x)=R(p,q)p^{-1/2}e^{\frac{i\pi}{4}sgn\ h}|h|^{1/2} sgn\ h-\frac{h}{2}+O(|h|^{1/2})
$$
as $h\to 0$ where $sgn\ h=\frac{h}{|h|}$ if $h\ne 0$, $sgn\ h=0$ if $h=0$, and $R(p, q)$ is a constant defined by
$$
R(p, q)=
\begin{cases}
(\frac{q}{p})e^{\frac{-\pi i}{4}(p-1)},\ \ \text{if}\ p\ \text{is odd and}\ q\ \text{even},\\
(\frac{p}{|q|})e^{\frac{\pi i}{4}q},\ \ \text{if}\ p\ \text{is even and}\ q\ \text{odd},\\
0,\ \  \text{if} \ p \ and\  q \text{ are odd},
\end{cases}
$$
with the Jacobi's symbol $(\frac{p}{q})$.
\end{theorem} 
\section{Quadratic reciprocity and Riemann's function}
Here we discuss the recent work of Chakraborty et.al \cite{K1} who gave a combined proof of both; that is the quadratic reciprocity law as well as the differentiability/non-differentiability of Riemann's function. 

Let $p$ be a natural number and $\mathfrak{z}=h+i\epsilon \in \mathcal{H}$  tending to $0$. We denote the upper half-plane by $\mathcal{H}$. Also let 
for $z \in \mathcal{H}\cup \mathbb{R}$,
\begin{equation*}\label{eqnrec5}
F(z)=\sum_{n=1}^{\infty}\frac{e^{\pi in^2z}}{\pi in^2}=\frac{1}{2}\sum_{\substack{n=-\infty\\  n\ne 0}}^{\infty}\frac{e^{\pi in^2z}}{\pi in^2}.
\end{equation*}
Let us denote by  $S(b,a)$  the quadratic Gauss sum defined by
\begin{equation*}\label{eqnrec8}
S(b,a)=\sum_{j=0}^{b-1}e^{2\pi i j^2 \frac{a}{b}}
\end{equation*}
for a natural number $b$.  One extends the definition for non-zero integral values  $b$ by,
\begin{equation*}\label{eqnrec8-1}
S(b,a)=S(|b|,\mathop{\rm sgn} (b)a).
\end{equation*}
We note that $S(|b|,-a)=\overline{S(|b|,a)}$ and $S(ka,kb)=S(a,b)$. 

We begin with the following result:
\begin{theorem}\label{thm1RZZ}
 For any integers $p>0$, $q$ we have
\begin{equation}\label{eqnrec9}
F\left(\frac{2q}{p}+\mathfrak{z} \right)-F\left(\frac{2q}{p}+i\epsilon \right)
=S(p,q)\frac{e^{-\pi i/4}}{p}\sqrt{\mathfrak{z}}-\frac{1}{2}h+O(\mathfrak{z}^2)
\end{equation}
where for a non-zero integer $p$, the coefficient is to be understood as 
$S(|p|,\mathop{\rm sgn}(p)q)$.
\end{theorem}
\begin{proof}
Let $b$ be an arbitrary real number.
 One can obtain by using  Euler-Maclaurin summation formula as in Lemma 4 in \cite{RM} (the resulting integral  can be evaluated as in \cite{KT} (Page 20--22)):
\begin{equation}\label{eqnrec1}
\sum_{n=-\infty}^{\infty}e^{{(b+pn)}^2i\mathfrak{z}}=
\frac{2\sqrt{\pi}}{p}e^{-\pi i/4}\sqrt{\mathfrak{z}}+O(\mathfrak{z})
\end{equation}
where the branch of $\sqrt{\mathfrak{z}}$ is chosen so that it is positive for $\mathfrak{z}>0$.

We integrate this along the line segment parallel to the real axis, say over $[\mathfrak{z}^\prime,\mathfrak{z}]$ with $\mathfrak{z}-\mathfrak{z}^{\prime}=h$. Now after separating the case $(b,n)=(0,0)$ the integrated form of \eqref{eqnrec1} becomes,
\begin{eqnarray}\label{eqnrec2}
h &+&\sum_{\substack{n=-\infty\\  (n,b)\ne (0,0)}}^{\infty}\frac{e^{{(b+pn)}^2i\mathfrak{z}}}{i{(b+pn)}^2} - \sum_{\substack{n=-\infty\\  (n,b)\ne (0,0)}}^{\infty}\frac{e^{{(b+pn)}^2i(h^\prime+i\epsilon)}}{i{(b+pn)}^2}\nonumber\\
&=&\frac{2\sqrt{\pi}}{p}e^{-\pi i/4}\sqrt{\mathfrak{z}}+O(\mathfrak{z}^2).
\end{eqnarray}
This can be re-written as
\begin{equation}\label{eqnrec3}
T(\mathfrak{z})-T(i\epsilon)
=\frac{2\sqrt{\pi}}{p}e^{-\pi i/4}\sqrt{\mathfrak{z}}-h(1+o(1))+O(\mathfrak{z}^2).
\end{equation}
Here
\begin{equation*}\label{eqnrec4}
T(\mathfrak{z})=T(\mathfrak{z},b)=\sum_{\substack{n=-\infty\\  (n,b)\ne (0,0)}}^{\infty}\frac{e^{{(b+pn)}^2i\mathfrak{z}}}{i{(b+pn)}^2}.
\end{equation*}
Then by the decomposition into residue classes, 
\begin{eqnarray*}\label{eqnrec6}
F\left(\frac{2q}{p}+\mathfrak{z} \right)&=&\frac{1}{2}\sum_{\substack{n=-\infty\\  n\ne 0}}^{\infty}\frac{e^{\pi in^2\left(\frac{2q}{p}+\mathfrak{z} \right)}}{\pi in^2}\nonumber\\
&=&\frac{1}{2}\sum_{b=0}^{p-1}e^{2\pi ib^2 \frac{q}{p}}\sum_{\substack{n\equiv b \pmod q\\  (n,b)\ne (0,0)}}^{}\frac{e^{\pi i n^2\mathfrak{z}}}{\pi in^2}\nonumber \\
&=&\frac{1}{2}\sum_{b=0}^{p-1}e^{2\pi ib^2 \frac{q}{p}}\frac{1}{\pi}T(\pi \mathfrak{z},b).
\end{eqnarray*}
Now using \eqref{eqnrec3},
\begin{equation}\label{eqnrec7}
\begin{split}
&F\left(\frac{2q}{p}+\mathfrak{z} \right)
=\frac{1}{2}\sum_{b=0}^{p-1}e^{2\pi ib^2 \frac{q}{p}}\frac{1}{\pi}\left(\frac{2\sqrt{\pi}}{p}e^{-\pi i/4}\sqrt{\pi\mathfrak{z}}-\frac{1}{2}h+T(i\epsilon,b)  \right)
+O(\mathfrak{z}^2) \\
&=\frac{1}{2}S(p,q)\left(\frac{2\sqrt{\pi}}{p}e^{-\pi i/4}\sqrt{\mathfrak{z}}-\mathfrak{z}  \right)+\frac{1}{2\pi}\sum_{b=0}^{p-1}e^{-2\pi ib^2 \frac{q}{p}}T(i\epsilon,b)+O(\mathfrak{z}^2) \\ \nonumber
&=S(p,q)\left(\frac{1}{p}e^{-\pi i/4}\sqrt{\mathfrak{z}}-\frac{1}{2}h  \right)+F\left(\frac{2q}{p}+i\epsilon \right)+O(\mathfrak{z}^2).
\end{split}
\end{equation}
\end{proof}
In \eqref{eqnrec2}, the variable can be $\frac{2q}{p}+\mathfrak{z}$ and $\frac{2q}{p}+\mathfrak{z}\prime$ and then instead of $h$ we would have $\mathfrak{z}-\mathfrak{z}\prime$. This will be used in deriving \eqref{eqnrec14-4}.

The relation  \eqref{eqnrec9} in this form is essentially Theorem 1 of S. Itatsu \cite{It} and from here non-differentiability of Riemann's function can be deduced.
Indeed, let $\mathfrak{z}=h+i\epsilon$ and let $\epsilon \to 0+$, in which we have to pay attention to the sign $\mathop{\rm sgn}h$ of $h$. Then

\begin{equation}\label{eqnrec9-1}
F\left(\frac{2q}{p}+h \right)-F\left(\frac{2q}{p} \right)
=S(p,q)\frac{e^{-\pi i/4\mathop{\rm sgn}h }}{p}\sqrt{|h|}-\frac{1}{2}h+O(h^2).
\end{equation}

Hence differentiability follows only in the case $S(p,q)=0$ with differential coefficient $-\frac{1}{2}$. This will be done in the next section appealing to Corollary \ref{cor1reciproc}. 
At the same time this is an elaboration of \cite[(47)]{RM} (on the right-hand side of which the factor $\sqrt{\pi}$ is to be deleted). Arguing as in \cite{RM} using the theta transformation formula, we may deduce the Landsberg-Schaar identity, from which the quadratic reciprocity may be deduced. 
 \begin{remark} We would like to make a few comments on the work of J. J. Duistermaat \cite{Du}.  
  In \cite[p. 4, $\ell\ell$. 1-2]{Du} J. J. Duistermaat says that ``this self-similarity formula was just an integrated version of the well-known transformation formula \eqref{eqnthetatrans1}.'' By this \cite[Theorem 4.2]{Du} is meant. 
  The equation (3.4) (was already proved by Cauchy \cite[pp.157-159]{Ca})  \cite{Du} for $r=\frac{q}{p}$ becomes 
 \begin{eqnarray*}\label{eqnduistermaat1-1}
\mu_\gamma(x)&=&e^{\frac{\pi}{4}m}p^{-\frac{1}{2}}{(x-r)}^{-\frac{1}{2}}\nonumber\\
&=&e^{\frac{\pi}{4}}p^{-1}S(2p,q){(x-r)}^{-\frac{1}{2}}.\nonumber
 \end{eqnarray*} 
 Incorporating this in \cite[(4.1)]{Du}, we see that it refers to the case $S(2p,q)$ of our Theorem \ref{thm1RZZ}. Hence by Corollary \ref{cor1reciproc}, differentiability of Riemann's function can be read off.
 
 Further 
 on \cite[p. 9, $\ell$ 7 from below]{Du} the relation (47) in \cite{RM} is stated in the form
\begin{equation*}\label{eqnduistermaat1}
\Theta\left(\frac{2q}{2p}+i\epsilon\right) \sim \frac{1}{p\sqrt{\epsilon}}S(2p,q), \quad \epsilon \to 0+ .
\end{equation*} 
Thus, we could say that \cite{Du} also gives material to deduce the reciprocity law.
 In \cite[Theorem 3.4]{Du} Duistermaat states that
\begin{equation}\label{eqnduistermaat2}
\Theta\left(z\right)=\begin{cases}
\Theta\left(\gamma z\right)e^{\frac{\pi i}{4}p}\left(\frac{-q}{p} \right)p^{-\frac{1}{2}}{(z-r)}^{-\frac{1}{2}} & p \; \text{odd} \\
\Theta\left(\gamma z\right)e^{\frac{\pi i}{4}(q+1)}\left(\frac{p}{|q|} \right)p^{-\frac{1}{2}}{(z-r)}^{-\frac{1}{2}} & q \; \text{odd}
\end{cases}
\end{equation} 
From \eqref{eqnduistermaat2}, the reciprocity law follows. However, it is used in its proof and thus unfortunately this does not lead to the proof of reciprocity law. 
\end{remark}
\subsection{Reciprocity law.}
The well-knwon law of quadratic reciprocity has had numerous proofs. Gauss, who first discovered the law, gave several proofs in his book, {\it Disquitiones Arithmeticae}. We recall the statement of the law of quadratic reciprocity. For a given pair of distinct primes $p$ and $q$, one can define the Legendre symbol $\big(\frac{p}{q}\big)$ to be $+1$ if the quadratic congruence $x^2\equiv p\pmod q$ has a solution; the symbol to be $-1$ if the quadratic congruence has no solution.
\begin{theorem}[Quadratic Reciprocity Law]
\begin{equation*}
\big(\frac{p}{q}\big)\big(\frac{q}{p}\big)=(-1)^{\frac{p-1}{2}.\frac{q-1}{2}}.
\end{equation*}
\end{theorem}
This theorem is remarkable in many ways, the most notable being the relationship between the solvability of the congruence $x^2\equiv q \pmod p$ to that of the congruence $x^2\equiv p\pmod q$. 
 Let us denote for $z \in \mathcal{H}$,
\begin{equation*}\label{eqnrec10}
\Theta(z)=\sum_{n=-\infty}^{\infty}e^{\pi in^2z}=1+2\sum_{n=1}^{\infty}e^{\pi in^2z}
\end{equation*}
and then the classical theta-function for $\mathop{\rm Re}z > 0$ is
\begin{equation*}\label{eqnrec11}
\theta(z)=\Theta(iz)=\sum_{n=-\infty}^{\infty}e^{-\pi n^2z}.
\end{equation*}
At this point, we note down the theta - transformation formula:
\begin{equation}\label{eqnthetatrans1}
\Theta(z)=e^{\frac{\pi i}{4}}z^{-\frac{1}{2}}\Theta\left(-\frac{1}{z} \right).
\end{equation}
We now prove the reciprocity law.
\begin{theorem}
Let $p \in \mathbb{N}$ and $(0\ne ) q \in \mathbb{Z}$. Then
 \begin{equation*}\label{eqnrec14-5}
S(p,q)=e^{\frac{\pi}{4}\mathop{\rm sgn}(q) i}{\left(\frac{p}{2|q|}\right)}^{1/2}S(4|q|,-\mathop{\rm sgn}(q)p).
 \end{equation*}
 \end{theorem}
 \begin{proof}
Let us first note that $F(z)$ is essentially the integral of $\Theta(z)$:
\begin{eqnarray}\label{eqnrec12}
\int_{0}^{z}\theta(-iz)\, {\rm d}z&=&\int_{0}^{z}\Theta(z)\, {\rm d}z \nonumber\\
&=&z+2\left(\sum_{n=1}^{\infty}\frac{e^{\pi in^2z}}{\pi in^2}-\sum_{n=1}^{\infty}\frac{e^{\pi in^2z}}{\pi in^2}\right)\nonumber\\
&=&z+2(F(z)-F(0)).
\end{eqnarray}
In particular, for $z=x+u+i\epsilon \in \mathbb{C}$ (with $\epsilon>0$) and $u \in (0,h)$, the above relation \eqref{eqnrec12} becomes
\begin{eqnarray}\label{eqnrec12-1}
\int_{x}^{x+h}\theta(\epsilon-iu)\, {\rm d}u&=&\int_{x+i\epsilon}^{x+h+i\epsilon}\Theta(z)\, {\rm d}z\nonumber\\
&=&h+2(F(x+h+i\epsilon)-F(x+i\epsilon)).
\end{eqnarray}
The theta-transformation formula \eqref{eqnthetatrans1} with $y>0$ gives
\begin{eqnarray*}\label{eqnrec13}
\theta(y-iu)&=&e^{\frac{\pi}{4}i}\frac{1}{\sqrt{u+iy}}\theta\left(\frac{i}{u+iy} \right)\nonumber\\
&=&e^{\frac{\pi}{4}i}\frac{1}{\sqrt{u+iy}}\sum_{n=-\infty}^{\infty}e^{\frac{i\pi n^2}{u+iy}}.
\end{eqnarray*}
We now make the following change of variable:
\begin{equation*}\label{eqnrec13-1}
\frac{i}{u+i\epsilon}=\frac{i}{x+v+i\epsilon}=\tau+\frac{1}{x}i
\end{equation*}
i.e., 
$$
\tau=\frac{\epsilon-iv}{x(x+v+i\epsilon)}\sim \frac{\epsilon-iv}{x^2}.
$$ 
Now with this change the integral in \eqref{eqnrec12-1} becomes
\begin{equation}\label{eqnrec12-2}
\int_{x}^{x+h}\theta(\epsilon-iu)\, {\rm d}u=-ie^{\pi/4}\int_{\frac{i}{x+h+i\epsilon}-\frac{i}{x}}^{\frac{i}{x+i\epsilon}-\frac{i}{x}}\frac{1}{{\left(\tau+\frac{1}{x}i \right)}^{\frac{3}{2}}}\theta\left(\tau+\frac{1}{x}i \right)\, {\rm d}\tau.
\end{equation}
The following relation is useful (which is in fact equivalent to \eqref{eqnrec12-1}) in applying integration by parts:
 \begin{equation*}\label{eqnrec12-1-0}
 \int_{}^{}\theta\left(\tau+\frac{i}{x}\right)\, {\rm d}u=\tau-2iF\left(-\frac{1}{x}+i\tau\right)+C.
 \end{equation*}
Using this, we may evaluate  \eqref{eqnrec12-2} and it becomes 
 \begin{eqnarray}\label{eqnrec14-1}
 &&\int_{\frac{i}{x+h+i\epsilon}-\frac{i}{x}}^{\frac{i}{x+i\epsilon}-\frac{i}{x}}\frac{1}{{\left(\tau+\frac{1}{x}i \right)}^{\frac{3}{2}}}\theta\left(\tau+\frac{1}{x}i \right)\, {\rm d}\tau \nonumber\\
 &=&\left[{(\tau+\frac{i}{x})}^{3/2}\left(\tau-2iF\left(-\frac{1}{x}+i\tau \right)\right)\right]_{\frac{i}{x+i\epsilon}-\frac{i}{x}}^{\frac{i}{x+h+i\epsilon}-\frac{i}{x}}\nonumber\\
&=&{\left(\frac{x+h+i\epsilon}{i}\right)}^{3/2}\left(\frac{i}{x+h+i\epsilon}-\frac{i}{x}-2iF\left(-\frac{1}{x+h+i\epsilon}  \right)  \right)\nonumber\\
&-&{\left(\frac{x+i\epsilon}{i}\right)}^{3/2}\left(\frac{i}{x+i\epsilon}-\frac{i}{x}-2iF\left(-\frac{1}{x+i\epsilon}  \right)  \right)+O(h).
 \end{eqnarray} 
 At this point, we note that 
 \begin{equation}\label{eqnrec14-2}
\frac{-1}{x+h+i\epsilon}=-\frac{1}{x}+\frac{1}{x^2}(\mathfrak{z}(1+o(1))).
 \end{equation}
Using \eqref{eqnrec14-2}, the main term in \eqref{eqnrec14-1} is
 \begin{eqnarray}\label{eqnrec14-3}
&&-2e^{\frac{\pi}{4} i}\left({\left(x+h+i\epsilon\right)}^{3/2} F\left(-\frac{1}{x+h+i\epsilon}  \right)-{\left(x+i\epsilon\right)}^{3/2}F\left(-\frac{1}{x+i\epsilon}  \right)\right) \nonumber\\
&=&2e^{\frac{\pi}{4} i}{\left(x+i\epsilon\right)}^{3/2}\left(F\left(-\frac{1}{x}+\frac{1}{x^2}\mathfrak{z}^\prime  \right)
-F\left(-\frac{1}{x}+\frac{1}{x^2}\epsilon^\prime \right)\right)+O(h),
 \end{eqnarray}
where we have used
$$
\mathfrak{z}^\prime=\mathfrak{z}(1+o(1)) ~~\mbox{and}~~\epsilon^\prime=\epsilon(1+o(1)).
$$
Now we specify $x=\frac{2q}{p}$
and apply Theorem \ref{thm1RZZ}. Under this specification \eqref{eqnrec14-3} takes the shape
 \begin{eqnarray*}\label{eqnrec14-4}
&=&2e^{\frac{\pi}{4} i}{\left(\frac{2q}{p}+i\epsilon\right)}^{3/2}\left(F\left(-\frac{2p}{4q}+{\left(\frac{p}{2q}\right)}^2\mathfrak{z}^\prime  \right)-F\left(-\frac{2p}{4q}+{\left(\frac{p}{2q}\right)}^2\epsilon^\prime \right)\right)
+O(h)  \nonumber\\
&=&2e^{\frac{\pi}{4} i}{\left(\frac{2q}{p}+i\epsilon\right)}^{3/2}S(4q,-p)e^{-\frac{\pi}{4} i}\frac{1}{4|q|}\big|\frac{p}{2q}\big|\sqrt{\mathfrak{z}^\prime}+O(h)
 \nonumber\\
&=&{\left(\frac{p}{2|q|}\right)}^{1/2}\frac{1}{p}S(4|q|,-\mathop{\rm sgn}(q)p)\sqrt{\mathfrak{z}^\prime}+O(h).
\end{eqnarray*}
Now on letting $\epsilon \to 0$, we get the desired result.
\end{proof}
We must have, correspondingly to \cite[(52)]{RM}
\begin{theorem}
For $p \in \mathbb{N}, 0\ne q \in \mathbb{Z}$, we have the reciprocity law
 \begin{equation}\label{eqnrec14-5}
S(p,q)=e^{\frac{\pi}{4}\mathop{\rm sgn}(q) i}{\left(\frac{p}{2|q|}\right)}^{1/2}S(4|q|,-\mathop{\rm sgn}(q)p).
 \end{equation}
 \end{theorem}
As a corollary we note that:
 \begin{corollary}\label{cor1reciproc}
 Let $x=\frac{q}{p}$ be of the form $\frac{2A+1}{2B+1}$, i.e. $p, q$ both being odd. Then
 \begin{equation}\label{eqnrec14-5-1}
R(2A+1,2B+1)=S(2p,q)=0
 \end{equation}
where $R$ is the coefficient in the forthcoming Eq. \eqref{eqnrec9-1-1}. 
 \end{corollary}
 \begin{proof}
 \begin{eqnarray*}\label{eqnrec14-5-2}
S(2p,q)&=&e^{\frac{\pi}{4} i}{\left(\frac{p}{2|q|}\right)}^{1/2}S(4|q|,2\mathop{\rm sgn}(q)p) \nonumber\\
&=&e^{\frac{\pi}{2} i}{\left(\frac{p}{2|q|}\right)}^{1/2}{\left(\frac{4|q|}{2|2p|}\right)}^{1/2}S(4\cdot 2p,2\mathop{\rm sgn}(q)|q|)\nonumber\\
&=&e^{\frac{\pi}{2} i}\frac{p}{|q|}\sqrt{\mathop{\rm sgn}(q)}S(2p,\mathop{\rm sgn}(q)|q|).
 \end{eqnarray*} 
 We now conclude \eqref{eqnrec14-5-1} by simply noting that $\mathop{\rm sgn}(q)|q|=q$. 
 \end{proof}
\begin{remark} 
 The relation \eqref{eqnrec14-5} leads to the so-called `Landsberg-Schaar' identity (see \cite[(5)]{RM}) if we take $p$ and  $q$ to be co-prime positive integers. This is
 \begin{equation*}\label{eqnrec14-6}
\frac{1}{\sqrt{p}}\sum_{j=0}^{p-1}e^{2\pi i j^2 \frac{q}{p}}=\frac{e^{\frac{\pi}{4}}i}{\sqrt{2q}}\sum_{j=0}^{2q-1}e^{2\pi i j^2 \frac{p}{2q}}.
\end{equation*}
\end{remark}
The following result will be required to complete the proof of the differentiability of Riemann's function.
\begin{lemma}
For a natural number $p$,
 \begin{equation*}\label{eqnrec9-1}
 S(p,q)=\varepsilon(p)\left(\frac{q}{p} \right)\sqrt{p}
 \end{equation*}
 where $\left(\frac{q}{p} \right)$ indicates the Jacobi symbol and 
 $$
 \varepsilon(p)=\begin{cases}
 1 & p\equiv 1\bmod 4  \\ i  & p\equiv 3\; \bmod 4
 \end{cases}
 $$
\end{lemma} 
 We are now ready to state a seemingly more general version of Theorem \ref{thm1RZZ}.  This  implies differentiability of Riemann's function at the rational point $\frac{2A+1}{2B+1}$ on putting $\mathfrak{z}=h+i\epsilon$ and $\epsilon \to +0$. We  note that similar result is also obtained in \cite[Theorem 4.2]{Du}.
 \begin{corollary}\label{cor2recipro}
 \begin{equation*}\label{eqnrec9-1}
 F\left(\frac{q}{p}+\mathfrak{z} \right)-F\left(\frac{q}{p}+i\epsilon \right)
 =R(p,q)\frac{e^{-\pi i/4}}{p}\sqrt{\mathfrak{z}}-\frac{1}{2}h+O(\mathfrak{z}^2),
 \end{equation*}
 where
 \begin{equation}\label{eqnrec9-1-1}
 \begin{split}
&R(p,2q)=S(p,q)=\varepsilon(p)\left(\frac{q}{p} \right)\sqrt{p}, \\
&R(2p,q)=S(4p,q)=e^{\frac{\pi}{4} i}\sqrt{2p}\left(\frac{-p}{q}  \right)\\
& R(2B+1,2A+1)=0.
 \end{split}
 \end{equation} 
 \end{corollary}
\begin{proof}
 Only the case $R(2p,q)$ needs to be considered (by Corollary \ref{cor1reciproc}).  Now by \eqref{eqnrec14-5} we have,
 \begin{equation*}\label{eqnrec9-1-2}
 \begin{split}
&R(2p,q)=S(4p,q)=e^{\frac{\pi}{4} i}{\left(\frac{4p}{2|q|}\right)}^{1/2}S(4|q|,-4\mathop{\rm sgn}(q)p) \\
&=e^{\frac{\pi}{4} i}{\left(\frac{2p}{|q|}\right)}^{1/2}S(|q|,-\mathop{\rm sgn}(q)p) \\
&=e^{\frac{\pi}{4} i}{\left(\frac{2p}{|q|}\right)}^{1/2}\sqrt{|q|}\varepsilon(|q|)\left(\frac{-\mathop{\rm sgn}(q)p}{|q|}  \right)\\
&=e^{\frac{\pi}{4} i}\sqrt{2p}\left(\frac{-p}{q}  \right).
\end{split}
 \end{equation*} 
\end{proof}

\begin{remark}
We make a historical remark on Riemann's function. \cite{BS} contains an almost complete list of references up to 1986. One addition is a correction to \cite{Sm} in 1983. After this, the review of \cite{G3} contains an almost complete list after \cite{BS} except for \cite{My} (esp. 619~) and \cite{Ul}. Among the papers listed in the review of \cite{G3}, we mention \cite{HT} and \cite{Jf} for consideration from the point of wavelets and \cite{Du} for self-similarity.
\end{remark}

\end{document}